\documentclass{amsart}
\usepackage{amssymb}

\title[Groups of Finite Morley Rank]{A Generic Identification Theorem
for Groups\\ of Finite Morley Rank, Revisited}

\author[A. Berkman]{Ay\c{s}e Berkman}
\address{Mathematics Department, Mimar Sinan University, Silahsor Cad. 71, Bomonti Sisli 34380, Istanbul, Turkey.}
\email{ayseberkman@gmail.com}

\author[A. V. Borovik]{Alexandre V. Borovik }
\address{School of Mathematics,  University of Manchester, Oxford Road, Manchester M13 9PL}

\email{alexandre.borovik@gmail.com}

\newtheorem{lemma}{Lemma}[section]
\newtheorem{theorem}[lemma]{Theorem}

\newtheorem{fact}[lemma]{Assertion}
\newcommand{\acfd}{algebraically closed field}
\newcommand{\acf}{algebraically closed field }

\newcommand{\fmr}{finite Morley rank }

\def\proof{\noindent {\bf Proof.} $\,$}

\begin{document}

\maketitle

\begin{abstract}
This paper contains a stronger version of a final identification theorem for the
`generic' groups of finite Morley rank.
\end{abstract}

\section{Introduction}

This paper contains a version of a generic identification theorem for simple groups of finite Morley rank adapted from the main result of \cite{BB04} by weakening its assumptions. It is published because it is being used in the authors' study of actions of groups of finite Morley rank \cite{BB11}.

A general discussion of the subject can be found in the books \cite{bn} and \cite{abc}.

A group of finite Morley rank is said to be of {\em $p'$-type}, if
it contains no infinite abelian subgroup of exponent $p$. Notice
that a simple algebraic group over an \acf $K$ is of $p'$-type if
and only if ${\rm char}\, K \ne p$. Other definitions can be found in
the next section.

The aim of this work is to prove the following.

\begin{theorem}
Let\/ $G$\/ be a simple  group of finite
Morley rank and\/ $D$ a maximal\/ $p$-torus  in $G$
of Pr\"{u}fer rank at least\/ $3$. Assume that
\begin{itemize}
\item[{\rm (A)}]  Every proper connected definable subgroup of $G$ which contains $D$ is a $K$-group. 

\item[{\rm (B)}] For every element\/ $x$ of order\/
$p$ in $D$, the group $C^\circ_G(x)$ is of\/ $p'$-type and
$C^\circ_G(x)=F^\circ(C^\circ_G(x))E(C^\circ_G(x)) $.

\item[{\rm (C)}]     $\langle C^\circ_G(x) \mid x \in D, \;|x|=p\rangle = G$.

\end{itemize}
Then $G$ is a Chevalley group over an \acf of characteristic distinct
from
$p$.
\label{sis-kebap}
\end{theorem}

Notice that, under assumption (B) of Theorem~\ref{sis-kebap},
$C^\circ_G(x)$ is a central product of $F^\circ(C^\circ_G(x))$
and $E(C^\circ_G(x))$.

The predecessor of this theorem, Theorem 1.2 of \cite{BB04}, was based on a stronger assumption than our assumption (A), namely, that \emph{every} proper definable subgroup of $G$ is a $K$-group, a condition that was difficult to check in its actual applications.

The proof of Theorem~\ref{sis-kebap} is given in
Section~\ref{sec:sis-kebap}.

\subsection{Definitions}

All definitions related to groups of finite Morley rank in general
can be found in \cite{bn} and \cite{abc}.

{}From now on $G$ is a group of finite Morley rank. The group $G$
is called 
%
a {\it $K$-group}, if every
infinite simple definable and connected section of the group is an
algebraic group over an algebraically closed field.
%
%


A {\em $p$-torus} $S$ is a direct product of finitely many copies
of the quasicyclic group ${{\Bbb Z}}_{p^\infty}$. The number of
copies is called the {\em Pr\"{u}fer $p$-rank} of $S$ and is
denoted by ${\rm pr}(S)$. For a definable group $H$, ${\rm
pr}(H)$ is the maximum of the Pr\"{u}fer ranks of $p$-subgroups
in $H$. It is easy to see that the Pr\"{u}fer $p$-rank of any
subgroup of a group of finite Morley rank is finite.

A group $H$ is called {\it quasi-simple} if $H'=H$ and $H/Z(H)$
is simple and non-abelian.  A quasi-simple subnormal subgroup of
$G$ is called a {\it component} of $G$.  The product of all
components of $G$ is called the {\it layer} of $G$ and denoted by
$L(G)$, and $E(G)$ stands for $L^\circ(G)$. It is known (see
Lemmas 7.6 and 7.10 in \cite{bn}) that $G$ has finitely many
components and that they are definable and are normal in $E(G)$.

$F(H)$ is the
{\em Fitting subgroup} of $H$, that is, the maximal normal definable
nilpotent subgroup.

If $H$ is a group of finite Morley rank then $B(H)$ is the
subgroup generated by all definable connected $2$-subgroups of
bounded exponent in $H$. Note that $B(H)$ is connected by
Assertion~\ref{zilb}.



\section{Background Material}

\subsection{Algebraic Groups}

For a discussion of the model theory of algebraic groups, the reader
might like to see Section 3.1 in \cite{berkman}.
The basic structural facts and definitions related to algebraic groups
can be easily found in the standard references
such as \cite{cartk,hump}.

First note that a connected algebraic group $G$ is called {\em
simple} if it has no proper normal connected and closed
subgroups. Such a group turns out to have a finite center, the
quotient group being simple as an abstract group. The classical
classification theorem for simple algebraic groups states that
simple algebraic groups over algebraically closed fields are
Chevalley groups, that is, groups constructed from Chevalley
bases in simple complex Lie algebras as described, for example, in
\cite{cartk}.

Now fix a maximal torus $T$ in a connected algebraic group
$G$ and denote the corresponding
root system by $\Phi$, and  for each $\alpha\in \Phi$, denote the
corresponding root subgroup by $X_\alpha$. The subgroup $\langle
X_\alpha,X_{-\alpha}\rangle$ is known to be isomorphic to $SL_2$
or $PSL_2$ and is called a {\em root $SL_2$-subgroup}.

If $G$ is simple,
the roots can have at most two different lengths, and the terms `short
root $SL_2$-subgroup' and `long root $SL_2$-subgroup'
have the obvious meanings.

A simple algebraic group is generated by its root $SL_2$-subgroups. In a
simple algebraic group, all long root
$SL_2$-subgroups are conjugate to each other, and similarly all short
root $SL_2$-subgroups are conjugate to each other.

\begin{fact}
\label{prop:class-in-algebraic} Suppose that\/ $G$ is a simple algebraic
group over an algebraically closed field. Let\/ $T$ be a
maximal torus in $G$ and $K$, $L$ closed subgroups of $G$
that are isomorphic to $SL_2$ or\/ $PSL_2$ and are
normalised by $T$. Then the following hold.
\begin{enumerate}
\item Either\/ $[K,L]=1$ or\/ $\langle
K,L\rangle$ is a simple algebraic group of rank $2$;
that is of type $A_2$, $B_2$ or\/ $G_2$.
\item The subgroups $K$ and $L$
are embedded in $\langle K,L\rangle$ as root $SL_2$-subgroups.
\item If $\langle K,L\rangle$ is of type $G_2$, then $G=\langle K,L\rangle$.
\end{enumerate}
\end{fact}

\paragraph{Proof.} The proof follows from the description of closed
subgroups in simple algebraic groups normalised by a maximal torus
\cite[2.5]{seitz}; see also \cite[Section~3.1]{seitz2}.   \hfill $\Box$

\begin{fact}
Let\/ $G$ be a simple algebraic group over an algebraically
closed field  of characteristic $\neq p$,
and let $D$ be a maximal $p$-torus in $G$. Then
$C_G(D)$ is a maximal torus in $G$. \label{fact:p-torus}
\end{fact}

\begin{proof} The proof follows from the description of
centralisers of  subgroups  of commuting semisimple elements in
simple algebraic groups \cite[Theorem~5.5.8]{ste1}.
\end{proof}

\subsection{Groups of Finite Morley Rank}

\begin{fact}
\label{zilb} {\rm (Zil'ber's Indecomposability Theorem)}
A subgroup of a
group of finite Morley rank which is generated by a family
of definable connected subgroups is also definable
and connected.
\end{fact}

\begin{proof} See \cite{zilber} or Corollary~5.28 in \cite{bn}.
\end{proof}


\begin{fact} {\rm \cite[Theorem~8.4]{bn}}
Let\/  $G\rtimes H$ be a group of finite Morley rank, where $G$
and\/ $H$ are definable, $G$ is an infinite simple algebraic group
over an \acf and\/ $C_H(G)=1$. Then $H$ can be viewed as a
subgroup of the group of automorphisms of\/ $G$, and moreover\/
$H$ lies in the product of the group of inner automorphisms and
the group of graph automorphisms of $G$ (which preserve root
lengths). In particular, when $H$ is connected, then $H$ consists
of inner automorphisms only. \label{defaut}
\end{fact}

\begin{fact}
{\rm \cite{ac}} \label{centext}
Suppose $G$ is a group of finite
Morley rank, $G=G'$, and\/ $G/Z(G)$ is a simple algebraic group
over an \acfd, and is of finite Morley rank, then $Z(G)$ is
finite and\/ $G$ is  also algebraic.
\end{fact}

\begin{lemma}
Let\/ $G$ be a connected  $K$-group of\/ $p'$-type  and\/ $D$  a
maximal\/ $p$-torus in $G$. If\/ $L \lhd G$ is a component in $G$,
then $D\cap L$ is a maximal\/ $p$-torus in $L$ and\/ $D=C_D(L)(D\cap L)$.
\label{lm:torus-in-component}
\end{lemma}

\begin{proof} As $G$ is connected, $L \triangleleft G$. Now the lemma
immediately follows from Assertions~\ref{centext}, \ref{defaut} and
\ref{fact:p-torus}. \end{proof}

\begin{lemma}
\label{algcomps} Under the assumptions of Theorem~{\rm
\ref{sis-kebap}}, we have, for every $p$-element\/ $t\in D$,
\[C^\circ_G(t)=F\cdot L_1\cdots L_r,\]
where $F=F^\circ(C_G^\circ(t))$ and each\/ $L_i$ is a simple
algebraic group over an algebraically closed field of
characteristic $ \ne p$.
\end{lemma}

\begin{proof} For every $p$-element $t$ in $G$, $C^\circ_G(t)=F\cdot
E(C^\circ_G(t))$. And assumption (C) ensures that
$C^\circ_G(t)$ is a $K$-group, hence its components are algebraic
groups by Assertion~\ref{centext}. \end{proof}

\subsection{Lyons's Theorem}

A detailed discussion of this particular version
of Lyons's theorem can be found in \cite{berkman}.

\begin{fact}[(Lyons \cite{gls,gls3})]
\label{Lyons} Suppose that\/ ${\Bbb F}$ is an algebraically closed
field, $I$ is one of the connected Dynkin diagrams of the simple
algebraic groups of Tits rank at least\/ $3$ and $\tilde{G}$ is
the simply connected simple algebraic group  of type $I$ over
${\Bbb F}$. Let\/ $G$ be an arbitrary group and for each $i\in I$,
$K_i$ stand for a subgroup of $G$ which is centrally isomorphic to
$PSL_2({\Bbb F})$, and\/ $T_i <  K_i$ denote a maximal torus in
$K_i$. Also assume that the following statements hold.
\begin{enumerate}
\item The group $G$ is generated by $K_i$ where $i\in I$.

\item For all\/ $i,j \in I$, $[T_i,T_j]=1$.

\item If\/ $i\ne j$ and $(i,j)$ is not an edge
in $I$, then $[K_i,K_j]=1$.

\item If\/ $(i,j)$ is a single edge in $I$, then
$G_{ij} =\langle K_i,K_j\rangle$ is isomorphic to
$(P)SL_3({\Bbb F})$.

\item If\/ $(i,j)$ is a double edge in $I$, then
$G_{ij}=\langle K_i,K_j\rangle$ is isomorphic to
$(P)Sp_4({\Bbb F})$.
Moreover, in that case, if\/ $r_i\in N_{K_i}(T_iT_j)$ and
$r_j\in N_{K_j}(T_iT_j)$ are involutions, then
the order of\/ $r_ir_j$ in $N_{G_{ij}}(T_iT_j)/T_iT_j$
is $4$.

\item For all $i,j\in I$, $K_i$ and $K_j$
are root\/ $SL_2$-subgroups of\/ $G_{ij}$
corresponding to the maximal torus $T_iT_j$ of\/ $G_{ij}$.
\end{enumerate}

Then there is a homomorphism from $\tilde{G}$ onto
$G$, under which the root $SL_2$-subgroups of $\tilde{G}$
{\rm (}for some simple root system{\rm )} correspond to the subgroups $K_i$.
\end{fact}

\subsection{Reflection Groups}

A linear semisimple transformation of finite order is called a
{\it reflection} if it has exactly one eigenvalue which is not 1.

\begin{theorem}
Let\/ $W$ be a finite group and assume that the following
statements hold.
\begin{enumerate}
\item There is a normal subset $S\subseteq W$ consisting of involutions
and generating $W$.
\item Over $\Bbb C$, $W$ has a faithful irreducible
representation of dimension $n\geqslant 3$ in which involutions from $S$ act as reflections.
\item For almost all prime numbers $q$, $W$ has faithful irreducible
representations {\rm(}possibly of different dimensions{\rm )} over
fields\/ ${\Bbb F}_q$. Moreover, for every such representation,
involutions in $S$ act as reflections.
\end{enumerate}
Then $W$ is one of the
groups $A_n$, $B_n$, $C_n$, $D_n$, $E_6$, $E_7$, $E_8$, $F_4$ for
$n\geqslant 3$.
\label{reflection}
\end{theorem}

\begin{proof} A proof, based on the classification of irreducible
complex reflection groups \cite{shephard-todd}, can be found in
\cite{berkman}. \end{proof}

\section{Proof of  Theorem~\ref{sis-kebap}}
\label{sec:sis-kebap}

The strategy is to construct the Weyl group and the root system
of $G$, and then to apply Lyons's Theorem. So from now on $G$ is
a simple group of \fmr and $D$ is a maximal $p$-torus in $G$
of Pr\"{u}fer rank $\geqslant 3$ and such that every proper definable connected subgroup contatining $D$ is a $K$-group. We also assume that
$C^\circ_G(x)$ is  of $p'$-type for every element $x \in D$ of
order $p$, $C^\circ_G(x)=F^\circ(C^\circ_G(x))E(C^\circ_G(x)) $
and
\[G = \langle C^\circ_G(x) \mid x \in D, \; |x|=p\rangle.\]

Notice that if $M$ is a proper definable subgroup in $G$ normalised by $D$ then $MD$ is a proper definable subgroup of $G$, for otherwise $M$ would be normal in $G$, which contradicts simplicity of $G$. Therefore, $MD$ and hence $M$ are  $K$-groups by assumption (A). This observation will be used throughout the proof.

We also shall systematically use the following observation.
\begin{lemma}
 $F^\circ(C^\circ_G(x))$ centralises
$D$ for every  element $x \in D$ of order $p$. \label{lm:F}
\end{lemma}

\begin{proof} The result immediately follows from the fact that a
$p$-torus in a definable nilpotent group belongs to the center of
this group \cite[Theorem~6.9]{bn}. \end{proof}

\subsection{Root Subgroups}

{}From now on, $SL_2$ will be used instead of $SL_2({\Bbb F})$, etc.
Denote by $\Sigma$ the set of all definable subgroups isomorphic
to $(P)SL_2$ and normalised by $D$. These are our future root
$SL_2$-subgroups.
If $N$ is a subgroup of $G$ which is normalised by $D$, then set
$H_N:=C_N(D\cap N)$. Note that if $K\in\Sigma$, then $H_K$ is
an algebraic torus in $K$.

\begin{lemma}
The set\/ $\Sigma$ is non-empty.
\end{lemma}

\begin{proof} Assume the contrary. If $L=E(C^\circ_G(x)) \ne 1$ for
some element $x$ of order $p$ from $D$, then, $L$ being a central
product of simple algebraic groups, contains a definable
$SL_2$-subgroup normalised by $D$. Therefore $C^\circ_G(x) =
F^\circ(C^\circ_G(x))$ centralises $D$ by Lemma~\ref{lm:F}. But
then \[G = \langle C^\circ_G(x) \mid x \in D, \; |x|=p\rangle\]
centralises $D$ which contradicts the assumption that $G$ is
simple.   \end{proof}

\begin{lemma}
\label{pairs} Let $K,L\in \Sigma$ be distinct and set
 $M = \langle K, L\rangle$. Then the following statements hold.
\begin{itemize}
\item[{\rm (1)}] The subgroup
$C_D(K) \cap C_D(L) \ne 1$ and $M$ is a $K$-group.

\item[{\rm (2)}] Either\/ $K$ and\/ $L$ commute or\/
$M$ is an algebraic group of type $A_2$, $B_2$ or $G_2$.

\item[{\rm (3)}]  $D\cap M = (D\cap K)(D\cap L)$ is a maximal\/ $p$-torus in
$M$.

\item[{\rm (4)}] If\/ $K$ and\/ $L$
do not commute then $H_M$ is a
maximal algebraic torus of the algebraic group $M$, and
$K$ and\/ $L$ are root\/ $SL_2$-subgroups of the algebraic
group $M$ with respect to the maximal torus $H_M$.

\item[{\rm (5)}] For all $K, L\in\Sigma$, we have $[H_K,H_L]=1$.

\item[{\rm (6)}]  For any $K,L\in\Sigma$, if the $p$-tori
$D\cap K$ and\/ $D\cap L$ have
intersection of order $>2$, then $K=L$.
\end{itemize}
\label{lm:tori}
\end{lemma}

\begin{proof} For the proof of $C_D(K) \cap C_D(L) \ne 1$ we refer the reader to the proof of \cite[Lemma~9.3]{berkman}. After that $M \leqslant C_G(C_D(K) \cap C_D(L))$ is a proper definable subgroup of $G$ and is a $K$-group since $D$ normalizes $M$.

(2)-(3) For  $L\in \Sigma$ set $R_L = C^\circ_D(L)$. If  $n =
{\rm pr}(D)$, then  the Pr\"{u}fer $p$-rank of $R_L$ is $n-1$.
Note that since $D$ is maximal and $D$ centralises a $p$-torus in
$L$, $D \cap L$ is a maximal $p$-torus in $L$. Now let $x$ be an
element of order $p$ in $R_K\cap R_L$. Then $K,L\leqslant
E(C_G(x))$ by the assumptions of the theorem. Set $E =
E(C_G(x))$. It follows from Lemma~\ref{lm:torus-in-component}
that the subgroup $D \cap E$ is a maximal $p$-torus of $E$, and
the subgroups $K$ and $L$, being $D$-invariant, lies in
components of $E$. If $K$ and $L$ belong to different components
of $E$, then they commute. Otherwise the component $A$ that
contains both $K$ and $L$ is a simple algebraic group, and
moreover $D\cap A$ is a maximal $p$-torus in $A$. Hence
the results follow from Assertion~\ref{prop:class-in-algebraic}.\\
(4)-(6) These follow by inspecting case by case and Assertion~\ref{prop:class-in-algebraic}. \end{proof}

\begin{lemma}
The subgroups in $\Sigma$ generate $G$.
\label{lm:sigmagenerates}
\end{lemma}

\begin{proof} Let $x\in D$ be of order $p$, then by assumption (B) of
the theorem
\[C^\circ_G(x) = F \cdot L_1\cdots L_n,\]
where $F = F^\circ(C^\circ_G(x))$ and  $L_i \lhd C^\circ_G(x)$ is
a simple algebraic group, for each $i=1,\ldots,n$.

The first step is to prove that $L_1\cdots
L_n\leqslant\langle\Sigma\rangle$. Note that $D\leqslant
C^\circ_G(x)$ and $D\cap L_i$ is a maximal $p$-torus in $L_i$ by
Lemma~\ref{lm:torus-in-component}. Let $H_i$ stand for the
maximal algebraic torus in $L_i$ containing $D\cap L_i$ and
$\Gamma_i$ be the collection of root $SL_2$-subgroups in $L_i$
normalised by $H_i$. Since $D\cap L_i\leqslant H_i$, $D\cap L_i$
normalises the subgroups in $\Gamma_i$. By Lemma~\ref
{lm:torus-in-component}, we have $D=C_D(L_i)(D\cap L_i)$, hence
$D$ normalises $\langle\Gamma_i\rangle=L_i$; that is
$\Gamma_i\subseteq \Sigma$ for each $i=1,\ldots,n$. This proves
the first step.

Hence for each $x\in D$ of order $p$, $C^\circ_G(x)=F\cdot
E(C^\circ_G(x))\leqslant F\langle\Sigma\rangle \leqslant
C_G(D)\langle\Sigma\rangle$. Therefore
\[G=\langle C^\circ_G(x)\mid x\in D, |x|=p\rangle\leqslant C_G(D)
\langle\Sigma\rangle.\] Since $C_G(D)$ normalises
$\langle\Sigma\rangle$, we have $\langle\Sigma\rangle\unlhd G$.
Now the result follows, since $G$ is simple. \end{proof}

We make $\Sigma$ into a graph by taking $SL_2$-subgroups $L \in
\Sigma$ for vertices and connecting two vertices $K$ and $L$ by
an edge if $K$ and $L$ do not commute.

\begin{lemma}
The graph $\Sigma$ is connected. \label{lm:connected}
\end{lemma}

\begin{proof} Otherwise consider a decomposition $\Sigma = \Sigma'
\cup \Sigma''$ of $\Sigma$ into the union of two non-empty sets
such that no vertex in $\Sigma'$ is connected to a vertex in
$\Sigma''$. Then we have
\[G = \langle \Sigma \rangle = \langle \Sigma' \rangle \times \langle \Sigma'' \rangle,
\]
which contradicts the assumption that $G$ is simple. \end{proof}

\begin{lemma}
If\/ $L \in \Sigma$ then $L = E(C_G(C_D(L)))$. \label{lm:L-unique}
\end{lemma}

\proof Let ${\rm pr}(D) = n$, then ${\rm pr}(C_G(C_D(L)))=n$
and ${\rm pr}(E(C_G(C_D(L)))) =1$. Since $L \leqslant
E(C_G(C_D(L)))$ and  $E(C_G(C_D(L)))$ is a central product of
simple algebraic groups over algebraically closed fields of
characteristic $\ne p$, we immediately conclude that $L =
E(C_G(C_D(L)))$. \hfill $\Box$

\subsection{Weyl Group}

Recall that when $L\in\Sigma$, $H_L$ stands for the maximal
algebraic torus $H_L:=C_L(D\cap L)$ in $L\cong SL_2$. Now set $H =
\langle H_L \mid L \in \Sigma \rangle$ and call it the {\em
natural torus associated with $D$}. It easily follows from
Lemma~\ref{lm:tori}(5) that $H$ is a divisible abelian group.

For any $L \in \Sigma$, $W(L):=N_L(H)H/H=N_L(H_L)/H_L$ is the
Weyl group of $SL_2$ and has order $2$; hence $W(L)$ contains a
single involution, which will be denoted by $r_L$. Note that the
 subgroup $L$ is uniquely determined by $r_L$, since $C_D(L) =
C^\circ_D(r_L)$ and $L = E(C_G(C_D(L)))$ by
Lemma~\ref{lm:L-unique}.

\begin{lemma}
\label{refgp} Consider a graph\/ $\Delta$ with the set of vertices
$\Sigma$, in which two vertices $K$ and\/ $L$ are connected by an edge
if\/ $[r_K, r_L] \ne 1$. If\/ $K$ and\/ $L$ belong to different
connected components of\/ $\Delta$, then $[K,L] =1$.

\end{lemma}

\proof It suffices to check this statement in the subgroup
$M=\langle K,L\rangle$, where it is obvious by Lemma~\ref{pairs}(2).
\hfill $\Box$

\bigskip
Notice that $D$ is a $p$-torus, the subgroups $N_G(D)$ and
$C_G(D)$ are definable and the factor group $N_G(D)/C_G(D)$ is
finite. Set $W:=N_G(D)/C_G(D)$. The images of involutions $r_L$,
for $L\in \Sigma$, in $W$ generate a subgroup which we denote by
$W_0$. Since, by their construction, involutions $r_L$ normalise
$D$, there is a natural action of $W_0$ on $D$.

\begin{lemma} The $p$-torus $D$ lies in the natural torus $H$.
In particular, $D$ is the Sylow $p$-subgroup of\/ $H$. \label{lm:D<H}
\end{lemma}

\proof Set $D' = \langle D\cap L \mid L \in \Sigma \rangle$. It suffices
to prove that $D'=D$. First note that $D'\leqslant D \cap H$.
If $D'<D$ then, since $[D, r_L] = D\cap L$,
all involutions $r_L$ act trivially on the factor group $D/D'$
which is divisible.
Let us take an element $d \in D$ which has sufficiently big order
so that the image of $d^{|W_0|}$ in $D/D'$ has order at least
$p^2$. Then the element $$z=\prod_{w\in W_0} d^w$$ has the same
image in $D/D'$ as $d^{|W_0|}$ and thus $z$ has order at least
$p^2$. Since $D = C_D(L)(D\cap L)$ and $|C_D(L) \cap (D\cap L)|
\leqslant |Z(L)| \leqslant 2$, we see that $|C_D(r_L):
C_D(L)|\leqslant 2$. Of course, the equality is possible only if
$p=2$. In any case, since $z \in C_D(r_L)$, $z^p \in C_D(L)$ for
all $L \in \Sigma$ and $z^p \ne 1$.  Hence $z^p \in C_G(\langle
\Sigma \rangle) = Z(G)$. This contradiction shows that $D = D'$
and $D \leqslant H$. \hfill $\Box$

\begin{lemma}
$N_G(D)= N_G(H)$. \label{lm:N=N}
\end{lemma}

\proof The embedding $N_G(H) \leqslant N_G(D)$ follows from
Lemma~\ref{lm:D<H}. Vice versa, if $x \in N_G(D)$ then the action
of $x$ by conjugation leaves the set $\Sigma$ invariant, hence it
leaves invariant the set of algebraic tori $\{H_L \mid L \in
\Sigma\}$ which generates $H$. Therefore $x \in N_G(H)$. \hfill
$\Box$

\begin{lemma}
$C_G(D)= C_G(H)$. \label{lm:C=C}
\end{lemma}

\proof Let $x \in C_G(D)$, then, for every $L \in \Sigma$, $x$
centralises $C_D(L)$ and thus, by Lemma~\ref{lm:L-unique},
normalises $L = E(C_G(C_D(L)))$. Since $x$ centralises a maximal
$p$-torus $D\cap L$ of $L$, it centralises the maximal torus $H_L
= C_L(D\cap L)$. Hence $x\in C_G(H)$. This proves $C_G(D)
\leqslant C_G(H)$. The reverse inclusion follows from
Lemma~\ref{lm:D<H}. \hfill $\Box$

\bigskip

In view of  Lemmata~\ref{lm:N=N} and \ref{lm:C=C}, we can refer to
$W_0$ either as the subgroup generated by the images of
involutions $r_L$ in the factor group $N_G(D)/C_G(D)$ or as the
subgroup generated by the images of involutions $r_L$ in
$N_G(H)/C_G(H)$. Also, we now know that the group $W_0$ acts on
$D$ faithfully.

\subsection{Tate Module}

Now the aim is to construct  a $\Bbb Z$-lattice on which $W_0$
acts as a crystallographic reflection group. For that purpose we
shall associate with $D$ the {\em Tate module} $T_p$. It is
constructed in the following way.

Let $E_{p^k}$ be the subgroup of $D$ generated by elements of order
$p^k$. Notice that every $r_L$ acts on $E_{p^k}$ as a reflection,
that is, $E_{p^k} = C_{E_{p^k}}(r_L)\times [E_{p^k}, r_L]$, $[E_{p^k},
r_L]\leqslant D \cap L$ is a cyclic group  and $r_L$
inverts every element in  $[E_{p^k}, r_L]$.

Consider the sequence of subgroups
\[E_p \longleftarrow E_{p^2} \longleftarrow E_{p^3} \longleftarrow
\cdots\] linked by the homomorphisms $x \mapsto x^p$. The
projective limit of this sequence is the free module $T_p$ over
the ring ${\Bbb Z}_p$ of $p$-adic integers. The action of $W_0$
on $D$ can be lifted to $T_p$, where it is still an irreducible
reflection group. By construction, $T_p /pT_p$ is isomorphic to
$E_p$ as a $W_0$-module. Notice also that $W_0$ acts on the
tensor product $T_p \otimes_{{\Bbb Z}_p} {\Bbb C}$ as a (complex)
reflection group, and that the dimension of $T_p
\otimes_{{\Bbb Z}_p} {\Bbb C}$ over $\Bbb C$
coincides with the Pr\"{u}fer
$p$-rank of $D$, hence is at least $3$.

\subsection{More Reflection Representations for ${W_0}$}

Now let us focus on odd primes $q\ne p$.
Consider the elementary abelian $q$-subgroups $E_q$ generated in
$H$ by all elements of the fixed prime order $q$.  For the sake of
complete reducibility of the action of $W$ on $E_q$, one can
consider only $q > |W|$.

Lemmas~\ref{kerdcc}, \ref{long} and \ref{equiv} below are
similar to some results in \cite{berkman}.
We include the proofs here for the sake of completeness of exposition.

\begin{lemma} {\rm \cite[Lemma 9.7]{berkman}}
\label{kerdcc} Let\/  $N=N_G(H)$, then $C_N(E_q) = C_G(H)$.
\end{lemma}

\proof It is clear that $C_G(H)\leqslant C_N(E_{q})$. To see the
converse, let $x\in C_N(E_{q})$. Since $x\in N$,
it acts on the elements of $\Sigma$ by conjugation.

First let us prove that $x$ normalises each subgroup in $\Sigma$.
To get a contradiction, assume that there is some subgroup
$L\in\Sigma$ such that $L^x\neq L$. But then by
Lemma~\ref{pairs}, $L$ and $L^x$ either commute or generate a
semisimple group as root $SL_2$-subgroups. Hence $|L\cap
L^x|\leqslant 2$. But this gives a contradiction since $q$ is an
odd prime and $L\cap E_{q}=L^x\cap E_{q} \leqslant L\cap L^x$.

Hence for each $L\in \Sigma$, $L^x=L$ and $x$ acts on $H\cap L$ as
an element from $N_L(H\cap L)$, since $SL_2$ does not have any
definable outer automorphisms. Note that the Weyl group of $SL_2$
is generated by an involution which inverts the torus $H\cap L$.
Since $x$ centralises $E_{q}\cap H$, $x$ centralises $H\cap L$ for
each $L\in\Sigma$, and hence $x$ centralises $H=\langle H\cap
L\mid L\in \Sigma\rangle$ and $x\in C_G(H)$. This proves the
equality. \hfill $\Box$

\bigskip

Now notice that $[E_q,r_L]$ is generated by a $q$-element in $H_L$
and thus has order $q$.  Hence $E_q$ is a finite dimensional
vector space over ${\Bbb F}_q$ on which $W_0$ acts as a group
generated by reflections.

\begin{lemma}
\label{wirr} The group $W_0$ acts irreducibly on $E_q$.
\end{lemma}

\proof Note that $W_0$ acts on $E_q$ faithfully by
Lemma~\ref{kerdcc}. Since $q > |W|$, the action of $W_0$ on $E_q$
is completely reducible. So if the action is reducible, then we can write
$E_q = E' \oplus E''$ for two proper $W_0$-invariant subspaces.

Assume that $W_0$ acts trivially on one of the subspaces $E'$ or
$E''$, say on $E'$. If $L\in \Sigma$, then $E_q = C_{E_q}(L)
\times (E_q \cap L)$, and, obviously, $C_{E_q}(L) = C_{E_q}(r_L)$.
Hence all $L \in \Sigma$ centralise $E'$ and $E' \leqslant
C_G(\langle \Sigma\rangle) = Z(G) =1$. Therefore $W_0$ acts
nontrivially on both $E'$ and $E''$.

For $L \in \Sigma$, the
$-1$-eigenspace $[E_q, r_L]$ of $r_L$ belongs to one of the
subspaces $E'$ or $E''$ and hence $r_L$ acts as a reflection on
one of the subspaces $E'$ or $E''$ and centralises the other.
Set $\Sigma'=\{\,L\in\Sigma \mid [E_q,r_L] \leqslant E'\,\}$
and $\Sigma'' = \{\,L\in\Sigma \mid [E_q,r_L] \leqslant E''\,\}$. It
is easy to see that $[r_K,r_L] =1$ for $K \in \Sigma'$ and $L \in
\Sigma''$. By Lemma~\ref{refgp}, $K$ and $L$ commute for all $K
\in \Sigma'$ and $L \in \Sigma''$, which contradicts
Lemma~\ref{lm:connected}. Hence $W_0$ is irreducible on $E_q$.
\hfill $\Box$

\begin{lemma}
The group $W_0$ acts irreducibly on $T_p \otimes_{{\mathbb{Z}}_p}
\mathbb{C}$.
\end{lemma}

\proof The proof is analogous to that of the previous lemma.
\hfill $\Box$

\subsection{Root System}

The aim of this subsection is to construct a root system on which $W_0$
acts as a crystallographic reflection group. The existence of such a
root system is guaranteed by the following lemma.

\medskip

\begin{lemma}
\label{rootsys} There exists an irreducible root system on which\/
$W_0$ acts as a crystallographic reflection group.
\end{lemma}

\proof Recall that $n \geqslant 3$. By Theorem~\ref{reflection},
the quotient group $W_0$ is one of the crystallographic reflection
groups $A_n, B_n, C_n, D_n, E_6, E_7, E_8, F_4$ and acts on the
corresponding root system. \hfill $\Box$

\bigskip

Let now $R= \{\, {r}_i \mid i \in I \,\}$ be a simple
system of reflections in $W_0$. We shall identify $I$ with the
set of nodes of the Dynkin diagram for $W_0$. It is well known
that every reflection in an irreducible reflection group $W_0$ is
conjugate to a reflection in $R$.

\begin{lemma} {\rm \cite[Lemma~9.9]{berkman}}
\label{long} Every reflection\/ $r \in W_0$ has the form $r_K$ for
some $SL_2$-subgroup $K \in \Sigma$.
\end{lemma}

\proof Let $ r\in  W_0$ be a reflection. Working our way back
through the construction of the module $T_p$, one can easily see
that the Pr\"{u}fer $p$-rank of $[D,r]$ is $1$.

Let $r_L\in W_0$ be a reflection which corresponds to a
$SL_2$-subgroup $L\in \Sigma$. By comparing the Pr\"{u}fer
$p$-ranks of the groups $C_D(r_L)$ and $C_D(r)$, we see that
$Z=(C_H(r_L) \cap C_H(r))^\circ$ has Pr\"{u}fer $p$-rank at least $1$.
Hence the subgroup $\langle L,H, r \rangle$ contains a non-trivial
central $p$-torus; also note that $\langle L,H, r \rangle$  is a
is a $K$-group, since $D$ lies in $H$. It is well known that a finite irreducible reflection
group contains at most two conjugate classes of reflections.
Therefore, after replacing $r$ and $r_L$ by their appropriate
conjugates in $W_0$, we can assume without loss of generality that
the images of $r_L$ and $r$ in $W_0$ correspond to adjacent nodes
of the Dynkin diagram. Now we can easily see that $\langle L,H, r
\rangle = Y * Z$ for some simple algebraic group $Y$ of Lie rank
$2$, and that $r=r_K$ for some root $SL_2$-subgroup $K$ of $Y$
such that $K\in \Sigma$. \hfill $\Box$

\subsection{Final Step}

The next task is to prove that the conditions of Lyons's Theorem
(Theorem~\ref{Lyons}) are satisfied.

\begin{lemma} {\rm \cite[Lemma~9.10]{berkman}}
\label{equiv} Attach an $SL_2$-subgroup $L_i \in \Sigma$ to each
vertex of the Dynkin diagram $I$ in such a way that the simple
reflection ${r}_i$ corresponding to this vertex is $r_{L_i}$ in
$W_0$. Then the following statements hold.

 {\rm (1)} $[L_i,L_j]=1$
if and only if\/ $| {r}_i {r}_j|=2$.

 {\rm (2)}   $\langle
L_i,L_j\rangle$ is isomorphic to $(P)SL_3$ if and only if $| {r}_i
{r}_j|=3$.

 {\rm (3)}   $\langle L_i,L_j\rangle$ is isomorphic to
$(P)Sp_4$ if and only if\/ $| {r}_i {r}_j|=4$.

 {\rm (4)}   $L_i$
and $L_j$ are embedded in $\langle L_i,L_j\rangle$ as root\/
$SL_2$-subgroups.
\end{lemma}

\begin{proof} It is well known that for each $i, j\in I$, the order
$| {r}_i {r}_j|$ takes the values $2$, $3$ or $4$, in a Dynkin
diagram of type $A_n, B_n, C_n, D_n, E_6, E_7, E_8$ or $F_4$. By
Lemma~\ref{pairs}, $L_i$ and $L_j$ either commute or generate
$(P)SL_3$, $(P)Sp_4$ or $G_2$. However $\langle
L_i,L_j\rangle\cong G_2$ is not possible in our case since $|
{r}_i {r}_j|=6$ does not occur in $I$.

The `only if' parts of (1) and (2) are easy to see.
In the case of part (3), that is when $L_i$ and $L_j$
generate $(P)Sp_4$, we have to show that $| {r_i} {r_j}|\neq 2$.
To get a contradiction, assume that $L_i$ and $L_j$
generate $(P)Sp_4$ and $| {r_i} {r_j}|=2$.
But then
 $L_i$ and $L_j$ are both
short root $SL_2$-subgroups. Note that ${r_i}$ and ${r_j}$ are
simple reflections, and it can be checked by inspection that one
of them must be a long reflection. This proves the `only if' part
of (3). Now parts (1), (2) and (3) follow from Lemma~\ref{pairs}
and the previous discussion. Part (4) is a direct consequence of
Lemma~\ref{pairs}. \end{proof}

\begin{lemma}
Each subgroup in $\Sigma$ is isomorphic to $(P)SL_2({\Bbb F})$ for
the same field $\Bbb F$.
\end{lemma}

\begin{proof} By Lemma~\ref{lm:connected} any two subgroups of
$\Sigma$ are connected by a sequence of edges. Note that each pair
which is connected by a single edge generates a simple group of
Lie rank 2 by Lemma~\ref{pairs}(2), hence their underlying fields
coincide. Thus the underlying fields of any two subgroups in
$\Sigma$ coincide. \end{proof}

Finally, we are in a position to apply Lyons's Theorem. Set $G_0$
to be the subgroup of $G$ generated by the subgroups $L_i$ for $ i
\in I$. By Lyons's Theorem, $G_0$ is a simple algebraic group over
$\Bbb F$ with the Dynkin diagram $I$. Its Weyl group, with respect
to the torus $T$, is $W_0$, hence $G_0$ contains all subgroups
from $\Sigma$. Therefore by Lemma~\ref{lm:sigmagenerates}, $G_0=G$
is a Chevalley group over $\Bbb F$. Since $G$ is of $p'$-type,
$\Bbb F$ is of characteristic different from $p$. \hfill $\Box$
$\Box$

\end{document}